\documentclass[12pt]{article}
\usepackage[english]{babel}
\usepackage[T1]{fontenc}
\usepackage{color}
\usepackage[centertags]{amsmath}
\usepackage{amsfonts}
\usepackage{amssymb}
\usepackage{amsthm}
\usepackage{newlfont}
\usepackage{amssymb,amsbsy,times,fancyhdr,color}
\usepackage{amsmath}
\usepackage[title]{appendix}
\usepackage{booktabs}
\usepackage{multirow}
\usepackage{subfigure}

\newcommand\ve{\varepsilon}

\newcommand\er{\mathbb{R}}
\newcommand\R{\mathbb{R}}

\newcommand\tl{\tilde}

\title{
 An HJB Approach to a General Continuous-Time Mean-Variance Stochastic Control Problem \footnote{{\bf ams classification:} 93E20, 60H10.}.\\
~ \\ G.~Aivaliotis\footnote{University of Leeds, School of
Mathematics, Leeds, LS2 9JT, UK, \, e-mail:
G.Aivaliotis@leeds.ac.uk} \quad \& \quad
A.~Yu.~Veretennikov\footnote{University of Leeds, School of
Mathematics, Leeds, LS2 9JT, UK, \, e-mail:
a.veretennikov@leeds.ac.uk \, \& \,University of Leeds, UK \& National Research University Higher School of Economics, Russian Federation \& Institute of Information Transmission
Problems, Moscow, Russian Federation} }

\begin{document}
\maketitle
\newcounter{count}[section]
\def\thecount{\thesection.\arabic{count}}
\newtheorem{Definition}[count]{DEFINITION}
\newtheorem{Theorem}[count]{THEOREM}
\newtheorem{Lemma}[count]{LEMMA}
\newtheorem{Corollary}[count]{COROLLARY}
\newtheorem{Remark}[count]{REMARK}
\newtheorem{Proposition}[count]{PROPOSITION}

\begin{abstract}

A general continuous mean-variance problem is considered for a diffusion controlled process where the reward functional has an integral and a terminal-time component. The problem is  transformed into a superposition of a static and a dynamic optimization problem. The value function of the latter can be considered as the solution to a degenerate HJB equation either in viscosity or in Sobolev sense (after a regularization) under suitable assumptions and with implications with regards to the optimality of strategies. There is a useful interplay between the two approaches -- viscosity and Sobolev. 
\end{abstract}

{\bf keywords:} mean-variance, stochastic control, Hamilton-Jacobi-Bellman, Sobolev solutions, viscosity solutions.\\ 

\section{Introduction}\label{Se1}

Mean-variance optimization problems have been established as a dominant methodology for portfolio optimization. Markowitz \cite{Mar} introduced the single-period formulation of the problem in 1952. It was not until the beginning of the new century, however, that dynamic mean-variance optimisation by means of dynamic programming received much attention, mainly due to the difficulties that the non-markovianity of the variance introduced to the problem. As an alternative to dynamic programming, the problem  was solved using martingale methods (see, e.g., Bielecki et al. \cite{bielecki2005}) or risk-sensitive functionals (see, e.g., Bielecki et al. \cite{bielecki2000}), whose second order Taylor expansion has the form of a mean-variance functional. 

A major advance in the theory for mean-variance functionals came by embedding the original problem into a class of auxiliary stochastic control problems that are in Linear-Quadratic form. This approach was introduced by Li and Ng \cite{li2000} in a discrete-time setting, while an extension of this method to a continuous-time framework is presented in Zhou and Li \cite{Zhou2000}, and further employed in Lim \cite{lim2004}. This approach leads to explicit solutions for the efficient frontier under some constraints imposed on the optimisation problem (they  assume that the reward function is a linear function of the controlled process).  Wang and Forsyth \cite{wang2010} design numerical schemes for auxiliary linear-quadratic problems formulated in  \cite{Zhou2000} and construct an efficient frontier. In \cite{Tse2013}, Tse et al. show that the numerical schemes designed in \cite{wang2010} provide indeed all the Pareto-optimal points for the efficient frontier.

Aivaliotis and Veretennikov \cite{A-Veretennikov2010} propose an alternative methodology that embeds the mean-variance problem into a superposition of a static and a dynamic optimisation problem, where the latter is suitable for dynamic programming methods. Solutions in the spaces of functions with generalised derivatives (henceforth called Sobolev spaces) are obtained through reqularisation. A further extension of this method is presented in \cite{A-Palczewski2014} where the viscosity solutions approach is followed. In the latter, each of the functionals either depends on the terminal value of the controlled process or on the integral from time $0$ to time $T$ of the controlled process are considered but separately. This approach does not in general provide any explicit solutions, but is geared towards numerical approximations that are proven to work efficiently. One advantage of the proposed methodology is that the problem can be solved for a pre-determined coefficient of risk aversion. For the LQ approach, the whole efficient frontier has to be traced and then optimal strategies can be assigned to different coefficients of risk-aversion. We should note that the strategies discussed in this paper are "precommitment" strategies. Alternatively one can consider "equilibrium" strategies (see Bj\"ork et al. \cite{Bjorketal2014}) or dynamically optimal strategies (see Pedersen and Peskir \cite{PedersenPeskir2017}) however these refer to different notions of optimality.

Let us consider a $d$-dimensional SDE driven by a $d$-dimensional
Wiener process $(W_t,{\cal F}_t, \, t\ge 0)$
\begin{equation}\label{diff}
dX_t= b(\alpha_t,t,X_t)\,dt+ \sigma(\alpha_t,t,X_t)\,dW_t, \quad
t\ge t_0, \qquad X_{t_0}=x.
\end{equation}
We will specify the assumptions on the coefficients $b$ and $\sigma$ later, depending on one or another approach that we take: based on Sobolev derivatives and solutions, or on viscosity solutions. The strategy $(\alpha_t, \, t_0\le t\le T)$ may be
chosen from the class ${\cal A}$ of all progressive measurable
processes with values in a compact convex  set $A \subset \mathbb{R}^\ell$. 
Admissible strategies are those for which the equation (\ref{diff}) has a unique  solution on $(\Omega, {\cal F}, \mathbb P, ({\cal F}_t))$. 
The second approach in this paper based on solutions of Bellman's equation in Sobolev spaces assumes that on our probability space there is another independent Wiener process $(\tilde W_t)$ of dimension $d$ such that the couple $(W_t, \tilde W_t)$ is $({\cal F}_t)$-adapted. 
 
~
 
We will use the standard
short notation where the dependence of $X$ on the strategy, initial
data $x$ and $t_0$ is shown by $E^\alpha_{t_0,x}$ in the
expectation; the full notation would be $X^{\alpha,t_0,x}_{t}$. Another important class of strategies ${\cal A}_M$ is a family of {\em feedback} ones -- also called Markov strategies -- given by an equality $\alpha_t = a(t,X_t)$ with some Borel measurable function $\alpha(\cdot)$ with values in $A$ such that there exists a (strong) solution of the equation 
\begin{equation}\label{diff1}
dX_t= b(\alpha(X_t),t,X_t)\,dt+ \sigma(\alpha(X_t),t,X_t)\,dW_t, \quad
t\ge t_0, \qquad X_{t_0}=x. 
\end{equation}
The issue of existence (and uniqueness) of solution for it has to be examined separately because the standard assumptions (see below) sufficient for the equation (\ref{diff}) may easily fail here.
This sub-class of strategies will  only be briefly mentioned in the last section with a proper reference for the interested reader; hence, we do not discuss how to tackle this problem here. Yet, note that in section \ref{sec:51} Markov strategies will generally speaking depend on time $t$ and on two state variables $x$ and an auxiliary $y$, not just on $(t,x)$.

Consider a (Borel measurable) instantaneous  reward function $f:\mathbb{R}^\ell\times[0,T]\times
\mathbb{R}^d\rightarrow\mathbb{R}$.  The (random) reward
from time $t_0$ to $T$ for a certain path of a process (\ref{diff})
and strategy $\alpha\in\mathcal{A}$ is expressed by the integral 
$\displaystyle \int_{t_0}^Tf(\alpha_s,s,X^{\alpha,t_0,x}_{s})\,ds$. At the terminal
time $T$, we will consider a ``final payment" $\Phi(X_T)$. Thus the
expected reward from $t_0$ to $T$ for a control strategy
$\alpha\in\mathcal{A}$ will be
$$
E_{t_0,x}^\alpha\left(\int_{t_0}^Tf(\alpha_s,s,X_s)\,ds+\Phi(X_T)\right)
.$$ 
In the mean-variance control problem, one aims at maximising the expected reward function  while penalizing for variance (that represents risk). For a control strategy $\alpha\in\mathcal{A}$ the functional is defined as 
\begin{align}\label{valuefunctional}
v^\alpha(t_0,x)&:=E^\alpha_{t_0,x}
\left(\int_{t_0}^T f(\alpha_s,s,X_s)\,ds+\Phi(X_T)\right)\nonumber\\
&-\theta\, \mathop{\mbox{Var}^\alpha_{t_0,x}}
\left(\int_{t_0}^T f(\alpha_s,s,X_s)\,ds+\Phi(X_T)\right),~~\theta\in\mathbb{R}.
\end{align}
$\theta>0$ indicates a risk averse investor, whereas $\theta<0$ a risk seeking investor. 
The {\em value function} is defined as a supremum, 
\begin{align}\label{valuefunction}
v(t_0,x)&:=\sup_{\alpha \in {\cal A}} v^\alpha(t_0,x).
\end{align} 
The class ${\cal A}_M$ is called sufficient for the control problem (\ref{diff}) \& (\ref{valuefunction}) iff $\displaystyle \sup_{\alpha \in {\cal A}} v^\alpha(t_0,x) = \sup_{\alpha \in {\cal A}_M} v^\alpha(t_0,x)$. 

In this paper we consider solutions to the problem of computing the value function via HJB equation both in Sobolev spaces and in viscosity sense and of finding an optimal or nearly optimal strategy for it. The contribution of this paper is twofold: we consider a general combined integral and terminal time payment functional and 
we discuss the optimality of strategies in different settings.
 
With regards to the functional, the problem when $\Phi(\cdot)=0$ has been discussed in \cite{A-Veretennikov2010} with solutions in Sobolev spaces. Both ($\Phi=0$, $f=1$) and    ($\Phi=1$, $f=0$) cases have been  discussed in \cite{A-Palczewski2014} using viscosity solutions. It is clear, however, that the solution to problem (\ref{valuefunction}) may not be derived as a combination of the previous two partial cases due to the nonlinearity because of the supremum involved. 

When looking for solutions in Sobolev spaces, we need to use regularisation (similar to \cite{A-Veretennikov2010}) as we cannot relax the non-degeneracy assumption. (This is not necessary for viscosity solutions; yet, the latter do not provide a clue for finding an optimal or nearly optimal strategy). We also do  relax the assumptions regarding the boundness of the drift and diffusion coefficients in the first setting based on viscosity approach,  
as well as of the reward function $f$ in comparison to the assumptions used in \cite{A-Veretennikov2010}. Finally we show that regularisation results into existence of  $\varepsilon-$optimal or nearly optimal strategies, whereas a verification theorem for viscosity solutions is only available  under rather strict boundness assumptions  which are not fulfilled in our context. Certain links between viscosity and Sobolev approaches are also shown with some useful consequences for both.

\section{Mean-Variance Control}\label{sec-mean-var}
The goal of this paper is to propose a way to compute a maximum of a  linear
combination of the mean and variance of a payoff function which involves both an integral and a final payment. 
The value function (\ref{valuefunction}) presents a genuinely non-markovian optimisation problem. This is due to the time-inconsistency of the variance term due to the square of the expectation and the square of an integral of the process. In detail
\begin{equation}
\begin{aligned}
v(t_0,x)&:=\sup_{\alpha \in {\cal A}} \Bigg\{E^\alpha_{t_0,x}
\Big(\int_{t_0}^T f(\alpha_s,s,X_s)\,ds+\Phi(X_T)\Big)\nonumber\\
&\hspace{45pt}-\theta\,\Big[E^\alpha_{t_0,x}\big(\int_{t_0}^T f(\alpha_s,s,X_s)\,ds+\Phi(X_T)\big)^2\nonumber\\
&\hspace{45pt}-\Big(E^\alpha_{t_0,x}\big(\int_{t_0}^T f(\alpha_s,s,X_s)\,ds+\Phi(X_T)\big)\Big)^2\Big]\Bigg\}.
\end{aligned}
\end{equation}

In order to deal with the square of the integral, we define the following state process $(X_t, Y_t)$ by the following stochastic
differential equation (as in Aivaliotis and Veretennikov \cite{A-Veretennikov2010} or Aivaliotis and Palczewski \cite{A-Palczewski2014}):
\begin{equation}\label{Yt}
\begin{aligned}
dX_t&= b(\alpha_t, t,X_t)\,dt + \sigma(\alpha_t, t,X_t)\,dW_t,~~X_{t_0}=x  \\ 
dY_t&= f(\alpha_t, t,X_t)\,dt,~~Y_{t_0}=y.
\end{aligned}
\end{equation}

The assumptions set later on will ensure existence and uniqueness of solutions to the above SDE. The different sets of assumptions, depending on the approach we follow,  will result in different types of solutions of the above SDE. We will comment on these in the relevant sections. Note that $f$ drives the dynamics of $Y_t$ in the extended state process $(X_t,Y_t)$; therefore, we will need to impose on $f$ the same assumptions which we impose on $b$.

Naturally, in order to write down a (backward) PDE, we have to allow the process $Y_t$ to depend on the initial data $Y_{t_0}=y\in \er$.
Then the value function can be written as $v(t_0,x):=\tilde{v}(t_0,x,0)$, where
\begin{equation}
\begin{aligned}
\tilde{v}(t_0,x,y)&=\sup_{\alpha \in {\cal A}} \Bigg\{E^\alpha_{t_0,x,y}
\Big(g^\alpha(X_T,Y_T)\nonumber\\
&\hspace{45pt}-\theta\,\Big[E^\alpha_{t_0,x,y}
\big(g^\alpha(X_T,Y_T)\big)^2
-\Big(E^\alpha_{t_0,x,y}
\big(g^\alpha(X_T,Y_T)\big)\Big)^2\Big]\Bigg\},
\end{aligned}
\end{equation}
with the terminal condition $\tilde v(T,x,y) = g^\alpha(x,y)=y+\Phi(x)$.


For the square of the expectation, we follow the dual representation $x^2 = \sup_{\psi \in \R} \{-\psi^2 - 2 \psi x\}$ (as in Aivaliotis and Veretennikov \cite{A-Veretennikov2010}). This results in the following representation:
\begin{equation}\label{vtilde}
\begin{aligned}
\tilde{v}(t_0,x,y)&=\sup_{\alpha \in {\cal A}} \Bigg\{E^\alpha_{t_0,x,y}
g(X_T,Y_T)-\theta\,E^\alpha_{t_0,x,y}\big(g(X_T,Y_T)\big)^2\nonumber\\
&\hspace{43pt}-\sup_{\psi\in\er}\Big\{-\theta\psi^2-2\theta\psi E^\alpha_{t_0,x,y} g(X_T,Y_T)\Big\}\Bigg\}\nonumber\\
&=\sup_{\psi\in\er}\Big\{V(t_0,x,y,\psi)-\theta\psi^2\Big\},\nonumber
\end{aligned}
\end{equation}
where $V(t_0,x,y,\psi)=\sup_{\alpha \in {\cal A}} E^\alpha_{t_0,x,y}\Big((1-2\theta\psi)g(X_T,Y_T)-\theta \big(g(X_T,Y_T)\big)^2\Big).$

\begin{Remark}
In principle, it is possible to deal with higher moments of the function $g$ by a linear approximation of the form $E(g^n)=\sup_{\psi,\phi}(\phi+\psi g)$. However one can show that the optimal $\phi$ will be the $(n-1)$th moment of $g$ making this approach impractical for $n>2$.  
\end{Remark}

\subsection{Viscosity solutions}
For an introduction to viscosity solutions for stochastic control problems, we refer the interested user to \cite[Chapter 4]{pham} or \cite{A-Palczewski2010}. In particular see \cite[Def. 4.2.1]{pham} for a definition of a viscosity solution.

In this section we make the following assumptions:
\begin{enumerate}
\item[{\bf (A$_{V}$)}]
\begin{itemize}
\item The functions $\sigma, b, f, \Phi$ are Borel with respect to $(a,t,x)$ and continuous with respect to $(a,x)$ for every $t$; moreover, there exist constants $K_1, K_2$ such that
$$
\begin{aligned}
&\|\sigma(a, t_1,x)-\sigma(a,t_2,z)\|\le K_1\left(\|x-z\|+|t_1-t_2|\right) \\
&\|b(a, t_1,x)-b(a, t_2,z)\|\le K_1\left(\|x-z\|+|t_1-t_2|\right)\\
&|f(a, t_1,x)-f(a,t_2,z)|\le K_2\left(\|x-z\|+|t_1-t_2|\right)
\end{aligned}
\quad \text{(Lipschitz condition)}
$$
\item
$$
\begin{aligned}
&\|b(a, t,x)\|\le K_1\big(1 + \|x\|\big)\\
&\|\sigma(a, t,x)\|\le K_1\big(1 +\|x\|\big)\\
&|f(a, t,x)|\le K_2\big(1 + \|x\|\big)
\end{aligned}
\hspace{100pt} \text{(linear growth condition)}
$$
\item  $|\Phi(x)| \le K_1(1 + \|x\|^m)$.
\end{itemize}
\end{enumerate}
For viscosity solutions we do not need to assume non-degeneracy of matrix $\sigma\sigma^T$. Note that the process $(X_t,Y_t)$ would have been strongly degenerate even if we had assumed non-degeneracy of $\sigma\sigma^T$. Under $(A_V)$ it follows from standard moment bounds on SDE solutions that the value function may grow at infinity no faster than some polynomial.  
\begin{Theorem}
Under assumptions (A$_V$) for every $\psi\in\er$ the value function $V(t_0,x,y,\psi)$ is a unique continuous polynomially growing viscosity solution of the following HJB equation:
\begin{equation}\label{eqn:HJB_visc}
\begin{cases}
V_{t_0} + \sup_{a \in A}\Big\{b(a, t_0,x)^T V_x +\frac{1}{2}tr\big(\sigma\sigma^T(a, t_0,x) V_{xx}\big)+f(u,t_0,x)V_{y}\Big\} = 0,&\\[5pt]
V(T, x, y, \psi) = (1 - 2 \theta\psi) g(x,y) - \theta \big(g(x,y)\big)^2,&
\end{cases}
\end{equation}
\end{Theorem}

\begin{proof}
We rewrite \eqref{eqn:HJB_visc} in a canonical form:
$$
\begin{cases}
- V_{t_0}(t_0,\tl x, \psi)  -H\Big(t_0,\tl x, V_{\tl x}(t_0,\tl x, \psi),  V_{\tl x\tl x}(t_0,\tl x), \psi\Big)=0,\\
 V (T, \tl x, \psi) =0,&
\end{cases}
$$
where $\tl x = (x,y)$, the Hamiltonian $H$ is given by
\begin{equation}\nonumber
H\Big(t,(x,y),\tl p,\tl M\Big)=\sup_{u \in A}\Big[b(u, t,x)^T p_1 +\frac{1}{2}tr\big(\sigma\sigma^T(u, t,x)M\big)+p_2f(u, t,x)\Big],
\end{equation}
with $\tl p = (p_1, p_2)$ and $M$ is obtained from $\tl M$ by removing the last row and column. Assumptions (A$_V$) imply that the domain of the Hamiltonian is the whole space ($dom(H)=\{(t,x,p,M)\in[0,T]\times\mathbb{R}^n\times\mathbb{R}^n\times\mathbb{S}^n\}$) and $H$ is continuous. By virtue of \cite[Theorem 4.3.1]{pham}, $V$ is a viscosity solution of \eqref{eqn:HJB_visc} (it is clearly of polynomial growth because $f$ satisfies the linear growth condition). Due to the Lipschitz property of $f$ and $\Phi$ the value function $V$ is continuous at the terminal time $t = T$. Hence, the comparison theorem (\cite[Theorem 4.4.5]{pham}) yields the continuity of $V$ and assures that $V$ is a unique continuous polynomially growing viscosity solution to \eqref{eqn:HJB_visc}.
\end{proof}

\begin{Remark}\label{Rem23}
Note that under assumptions $(A_V)$ the value function is the same for any filtration $({\cal F}_t)$. This will be used in the section \ref{sec4} in Theorem \ref{thmcombined}.
\end{Remark}
We will need a similar result about the regularised version of the state process $(X_t, Y_t)$ given by the SDE system 
\begin{align}\label{Yt-epsilon}
dX_t&= b(\alpha_t, t,X_t)\,dt + \sigma(\alpha_t, t,X_t)\,dW_t, \quad X_{t_0}=x, 
\nonumber \\ \\ \nonumber 
dY_t^\ve&= f(\alpha_t, t,X_t)\,dt+\ve d\tilde{W}_{t}, \quad Y_{t_0}=y.
\end{align}
Here $\tilde W$ is a $d$-dimensional Wiener process independent of $W$. As was prompted earlier, we assume that the pair $(W_t, \tilde W_t)$ is adapted to the filtration $(\cal F_t)$. 
Accordingly we define the regularized value function: 
\begin{align}
\tilde{v}^\ve (t_0,x,y)&=\sup_{\alpha \in {\cal A}} \Bigg\{E^\alpha_{t_0,x,y}
g(X_T,Y_T^\ve)-\theta\,E^\alpha_{t_0,x,y}\big(g(X_T,Y_T^\ve)\big)^2\nonumber\\
&\hspace{43pt}+\sup_{\psi\in\er}\Big\{-\theta\psi^2-2\theta\psi E^\alpha_{t_0,x,y} g(X_T,Y_T^\ve)\Big\}\Bigg\}\nonumber\\
&=\sup_{\psi\in\er}\Big\{V^\ve(t_0,x,y,\psi)-\theta\psi^2\Big\},\nonumber
\end{align}

\noindent  where $V^\ve(t_0,x,y,\psi)=\sup_{\alpha \in {\cal A}} E^\alpha_{t_0,x,y}\Big((1-2\theta\psi)g(X_T,Y_T^\ve)-\theta \big(g(X_T,Y_T^\ve)\big)^2\Big).$

\begin{Theorem}
Under assumptions (A$_V$) for every $\psi\in\er$ the value function $V^\ve(t_0,x,y,\psi)$ is a unique continuous polynomially growing viscosity solution of the following HJB equation:
\begin{equation}\label{eqn:HJB_term}
\begin{cases}
V_{t_0} + \sup\limits_{a \in A}\Big\{b(a, t_0,x)^T V_x +\frac{1}{2}tr\big(\sigma\sigma^T(a, t_0,x) V_{xx}\big)+f(u,t_0,x)V^\ve_{y} + \frac{\ve^2}{2}V^\ve_{yy}\Big\} = 0,&\\[5pt]
V(T, x, y, \psi) = (1 - 2 \theta\psi) g(x,y) - \theta \big(g(x,y)\big)^2,&
\end{cases}
\end{equation}
\end{Theorem}
The proof is similar. Of course, for the regularised system solution of the HJB equation also exists in the Sobolev sense which is the content of the next section.

\section{Sobolev Solutions}

In this section we suggest suitable HJB equations for the mean-variance problem, as reformulated in the previous section. The solutions of parabolic HJBs will be considered in
the Sobolev classes $W^{1,2}_{p, loc}$ with one derivative with
respect to $t$ and two with respect to $x$ in $L_p$ in any bounded domain. 
Denote $\overline W^{1,2}_{loc} = \bigcap_{p>1} W^{1,2}_{p, loc}\bigcap C$. For the
functions of three variables, $v(t,x,y), \, 0\le t\le T, \, x,y\in
R^d, \, $ we will use the Sobolev class $\overline W^{1,2,2}_{loc} =
\bigcap_{p>1} W^{1,2,2}_{p, {loc}}\bigcap C$. To ensure uniqueness of solutions of the forthcoming Bellman's equations, we will be looking for these solutions growing at infinity no faster than some polynomial. These classes will be denoted, respectively, by $\overline {CW}^{1,2}_{loc, poly}$ and $\overline {CW}^{1,2,2}_{loc, poly}$.

Throughout this section, we assume the following 

\begin{enumerate}
\item[\bf{A$_S$}]
\begin{itemize}
\item The functions $\sigma, b, f$ are Borel with respect to $(\alpha,t,x)$,
continuous with respect to $(\alpha,x)$ and continuous with respect to
$x$ uniformly over $u$ for each $t$. $\Phi(x)$ is continuous with respect to $x$. Moreover, there are constants $K_1, K_2$ such that
\item $\|\sigma(\alpha,t,x)-\sigma(\alpha,t,x')\|\le K_1|x-x'|$,
\item $\|b(\alpha,t,x)-b(\alpha,t,x')\|\le K_1|x-x'|$,
\item $|f(\alpha,t,x)-f(\alpha,t,x')|\le K_2|x-x'|$,
\item $\|\sigma(\alpha,t,x)\| + \|b(\alpha,t,x)\|\le K$, 
\item $|f(\alpha,t,x)|\le K_2$, 
\item  $|\Phi(x)| \le K_2(1 + \|x\|^m)$,
\item $\sigma\sigma^T$ is uniformly non-degenerate.
\end{itemize}
\end{enumerate}

In order to establish the existence of solutions in Sobolev spaces, it is essential that the resulting HJB equations are non-degenerate. Yet, the state process \eqref{Yt} is strongly degenerate and so will be the resulting HJB equation for problem (\ref{vtilde}). In order to avoid degeneracy, apart from assuming non-degeneracy for $\sigma\sigma^T$ we add a small constant positive diffusion coefficient $\ve>0$ with an independent (to $W_t$) Wiener process to the  variable $Y_t$ in \eqref{Yt}. The regularised state process $(X_t, Y_t)$ has been introduced in ghe previous section in the equation (\ref{Yt-epsilon}). 

\begin{Theorem}
Under  assumptions (A$_S$) for every $\psi\in\er$ the value function $V^\ve(t_0,x,y,\psi)$ is a unique solution in $\overline {CW}^{1,2,2}_{loc, poly}$ of the  HJB equation (\ref{eqn:HJB_term}). 
\end{Theorem}
\begin{proof}
Under Assumptions (A$_S$) the result follows from \cite[Chapters 3 and 4]{kry}. 
\end{proof}

In the next section, we will show that the function $V^{\varepsilon}$
is locally Lipschitz in $\psi$ and grows at most linearly in this
variable. Hence, the supremum is again attained at some $\psi$ from
a closed interval. Then the external optimisation problem becomes:
\begin{equation}\label{extern1fin}
{v}^{\varepsilon}(t_0,x,y)=\sup_\psi\left[
V^{\varepsilon}(t_0,x,y,\psi)-\theta\psi^2\right].
\end{equation}

\section{Properties of value functions}\label{sec4}

In this sections we show some properties of the value functions that are common in both approaches described above. These are important properties that allow the numerical solution of the mean-variance problem to be tractable. We assume that a new set of assumptions (A$_0$) holds, which is the union (mathematically intersection: both of them are satisfied) of the assumptions (A$_S$) and (A$_V$).

\begin{Theorem}\label{thmcombined}
Under Assumptions (A$_0$), viscosity and Sobolev solutions of the HJB equation coinside. In particular,  the Sobolev solution coinsides with the value function computed for the class of strategies adapted to any filtration $({\cal F}_t)$ which means that the value function does not depend on the particular filtration. 
\end{Theorem}
In a way, this is a repetition of the Remark \ref{Rem23}.

\begin{Theorem}\label{thm:cont_dep_psi}
$\ $ Under Assumptions (A$_0$):
\begin{itemize}
\item[i)] The functions $V, V^\ve,$ ($V^{(\ve)}$ in short which stands either for $V$, or for $V^\ve$ in the sequel) are continuous in $\psi$ and convex in $\psi$. If $f, \Phi$ are non-negative, $V, V^\ve$ are decreasing in $\psi$.
\item[ii)] There exists a constant $C$ such that
\begin{equation}\nonumber
| V^{(\ve)}(t_0,x,y,\psi)- V^{(\ve)}(t_0,x,y,\psi')| \le C\, (1 + \|x\|)\, |\psi-\psi'|.
\end{equation}

\item[iii)]The value function $v$ is given by
\begin{equation*}
v(t_0, x) = \sup_{\psi_{min}\le\psi \le \psi_{max}} \big\{  V^{(\ve)}(t_0, x, 0, \psi) - \theta \psi^2 \big\},
\end{equation*}
where
$$
\psi_{min}=-\sup_{\alpha\in\mathcal{A}}E_{t_0,x,y}^\alpha g(t,X_t,Y_t) > -\infty, 
$$ 
and
$$
\psi_{max}=-\inf_{\alpha\in\mathcal{A}}E_{t_0,x,y}^\alpha g(t,X_t,Y_t) < +\infty.
$$ 
\end{itemize}
\end{Theorem}

\begin{proof}
The proof follows the same line of reasoning as in \cite[Theorem 2.2]{A-Palczewski2014} and making use of the definition for the function $g$. For part (i), it is straightforward to prove convexity, which implies continuity with respect to $\psi$ where finite (recall that our function under consideration is finite everywhere, though). It is clear from the definition of $V, V^\ve$ that they are decreasing in $\psi$ for non-negative $f, \Phi$. For part (ii) we check that $V^\ve$ grows at most linearly in $\psi$, which implies that the mapping $h(\psi) =  V^\ve(t_0, x, y, \psi) - \theta \psi^2$ attains its maximum in a compact interval. By convexity, $V^\ve$ has well-defined directional derivatives. Hence, $h$ also has well-defined directional derivatives and in a point where the maximum is attained the left-hand side derivative is non-negative while the right-hand side derivative is non-positive. We then show that $\partial^+ h(\psi) > 0$ for $\psi < \psi_{min}$ and $\partial^- h(\psi) < 0$ for $\psi > \psi_{max}$. This implies that the conditions for maximum can only be satisfied in the interval $[\psi_{min}, \psi_{max}]$.
We skip further details and refer the reader to \cite[Theorem 2.2]{A-Palczewski2014}.
\end{proof}

\begin{Theorem}\label{Thm6fin}
Under assumptions (A$_0$):
\begin{equation}\label{boundfin}
\sup_{t,x,y}\, | v^{\varepsilon}(t_0,x,y) - v(t_0,x,y)| \le |\theta|(\varepsilon^2 (T-t_0)+\varepsilon C_T\sqrt{T-t_0}).
\end{equation}
\end{Theorem}
\begin{proof}  We
have,
\begin{align}\label{e10}
&|v^{\varepsilon}(t_0,x,y)-v(t_0,x,y)|\nonumber\\ \nonumber
\\& \le \sup_\psi |\big(V^{\varepsilon}(t_0,x,y,\psi )-\theta
\psi^2-V(t_0,x,y,\psi)+\theta \psi^2\big)| \nonumber\\
&= \sup_\psi|\Big\{\sup_{\alpha\in \cal{A}}E_{t_0,x,y}^\alpha\Big(g(X_T,Y_T^\ve)[1-2\theta \psi]-\theta
\big(g(X_T,Y_T^\ve)\big)^2\Big)
\nonumber\\
&\hspace{3cm} -\sup_{\alpha\in
{\cal A}}E_{t_0,x,y}^\alpha \Big(g(X_T,Y_T)[1-2\theta \psi]-\theta
\big(g(X_T,Y_T)\big)^2\Big)\Big\}| \nonumber\\
&\le\sup_\psi\sup_{\alpha\in \cal{A}}|E_{t_0,x,y}^\alpha\Big\{\big(g(X_T,Y_T^\ve)-g(X_T,Y_T)\big)[1-2\theta \psi]
\nonumber\\
&\hspace{6cm}
-\theta\Big(\big(g(X_T,Y_T^\ve)\big)^2-
\big(g(X_T,Y_T^\ve)\big)^2\Big)\Big\}|\nonumber\\
&=\sup_\psi\sup_{\alpha\in \cal{A}} |E_{t_0,x,y}^\alpha\Big\{[1-2\theta\psi]\int_{t_0}^T\ve\,d\tilde W_t-\theta\big(g(X_T,Y_T^\ve)
\nonumber\\
&\hspace{5cm}
-g(X_T,Y_T)\big)\big(g(X_T,Y_T^\ve)+g(X_T,Y_T)\big)\Big\}|\nonumber\\
&=\sup_\psi\sup_{\alpha\in \cal{A}} |E_{t_0,x,y}^\alpha\Big\{[1-2\theta\psi]\int_{t_0}^T\ve\,d\tilde W_t-\theta\Bigg(2\int_{t_0}^T\ve\,d\tilde W_t\int_{t_0}^Tf(t,X_t)\,dt\nonumber\\
& \hspace{5cm}
+\Big(\int_{t_0}^T\ve\,d\tilde W_t\Big)^2+2\Phi(X_T)\int_{t_0}^T\ve\,d\tilde W_t\Bigg)\Big\}|\nonumber\\
&\le |\theta|(\varepsilon^2 (T-t_0)+\varepsilon C_T\sqrt{T-t_0}):=C_{\varepsilon,\theta,T}.
\end{align}
The last inequality comes from the growth assumptions on (A$_0$) and the use of 
Cauchy--Bunyakovsky-Schwarz inequality along with  moment estimates for $X_T$ and $Y_T$. We add that the reason for using Cauchy--Bunyakovsky-Schwarz inequality is dependence of any  strategy $\alpha$ on $\tilde W$ for any $\varepsilon>0$ and, hence, dependence of $X$ and $\tilde W$. 
\end{proof}
Clearly, the power function $(T-t_0)^{1/2}$ in (\ref{e10}) can be replaced by any $(T-t_0)^\beta$ with $0<\beta<1$ at the expence of the constant $C_T$, if we use H\"older's inequality instead.

\section{Optimal Strategies}

\subsection{Sobolev approach}\label{sec:51}
While working with solutions of HJB equations in Sobolev spaces, there is a verification theorem that ensures the optimality of the strategy that corresponds to the particular solution, or, at least a nearly optimality of a smoothed version of such a strategy; here smoothing is available due to the convexity of the set $A$. 
Since the initial problem has been regularised, we can now only hope for ``almost-optimal'' strategies for the original problem.

\begin{Definition}
Let $\delta\ge 0$. A strategy $\alpha\in \mathcal{A}$ is said to
be $\delta-$optimal for $(t,x)$ if $v(t,x)\le
v^\alpha(t,x)+\delta$, where
$v(t,x)=\sup_{\alpha\in\mathcal{A}}
v^\alpha(t,x)$. We say that there exists a nearly optimal strategy iff a $\delta$-optimal strategy can be found for any $\delta>0$.
\end{Definition}
\begin{Lemma}\label{lestrat}

Under assumptions (A$_0$), for any strategy $\alpha\in\mathcal{A}$, we have the following
bounds
\begin{equation}\label{epszone}
|v^{\varepsilon,\alpha}(t_0,x,y)-v^{\alpha}(t_0,x,y)|\le
C_{\varepsilon,\theta,T}.
\end{equation}
\end{Lemma}
\begin{proof}
Quite similarly to Theorem \ref{Thm6fin}, by the same calculus without $\sup_{\alpha}$ one can show  that 
\begin{equation}\nonumber |v^{\varepsilon,\alpha}(t_0,x,y)-v^{\alpha}(t_0,x,y)|\le C_{\varepsilon,\theta,T}.\end{equation}
The only difference is that now the bounds are written for a fixed particular strategy. 
Therefore, 
\begin{equation}
v^{\varepsilon,\alpha}(t_0,x,y) - C_{\varepsilon,\theta,T}\le
v^{\alpha}(t_0,x,y)\le v^{\varepsilon,\alpha}(t_0,x,y) +
C_{\varepsilon,\theta,T}, 
\end{equation}
as required.
\end{proof}

\begin{Theorem}\label{Thm53}
Assume (A$_0$). Let the strategy $\bar\alpha^\varepsilon\in\mathcal{A}$ be an 
optimal  strategy for the problem (\ref{extern1fin}), or if the
supremum is not attained, let the strategy
$\tilde\alpha^\varepsilon\in\mathcal{A}$ be a $\kappa-$optimal
strategy for the same problem . Then, the same strategy is
$\delta$-optimal for the original degenerate value function with
appropriate choice of the constant $\delta$, so that $\delta \to 0$ as $\varepsilon, \kappa \to 0$, i.e., there exists a nearly optimal strategy.

\end{Theorem}
\begin{proof} Suppose that the value function of the degenerate
problem attains its supremum for a strategy $\bar \alpha \in {\cal A}$. Then, we
would have from Lemma \ref{lestrat}, 
\begin{equation}
v^{\bar\alpha}(t_0,x,y) - C_{\varepsilon,\theta,T}\le
v^{\varepsilon,\bar\alpha}(t_0,x,y)\le v^{\bar\alpha}(t_0,x,y) +
C_{\varepsilon,\theta,T}.
\end{equation}
Furthermore, because the strategy $\bar\alpha^\varepsilon$ is
optimal for the regularised value function (assuming that the
supremum is attained), we know that
\begin{equation}
v^{\varepsilon,\bar\alpha^\varepsilon}(t_0,x,y)\ge
v^{\varepsilon,\bar\alpha}(t_0,x,y).
\end{equation}
So,
\begin{equation}
v^{\bar\alpha}(t_0,x,y) - C_{\varepsilon,\theta,T}\le
v^{\varepsilon,\bar\alpha}(t_0,x,y)\le
v^{\varepsilon,\bar\alpha^\varepsilon}(t_0,x,y)\le
v^{\bar\alpha^\varepsilon}(t_0,x,y) + C_{\varepsilon,\theta,T}.
\end{equation}
Hence
\begin{equation}
v^{\bar\alpha}(t_0,x,y) - 2C_{\varepsilon,\theta,T}\le
v^{\bar\alpha^\varepsilon}(t_0,x,y),
\end{equation}
i.e. $\bar\alpha^\varepsilon$ is $2C_{\varepsilon,\theta,T}-$optimal for
$v^\alpha$ or $\delta=2C_{\varepsilon,\theta,T}$ in this case.

In the case if the degenerate value function does not attain a
supremum, a $\gamma-$optimal strategy exists for any $\gamma>0$; let us denote any such strategy by  $\tilde\alpha$.
Then
\begin{equation}\label{almostopt}
v^{\tilde\alpha}(t_0,x,y)\ge
\sup_{\alpha\in\mathcal{A}}v^{\alpha}(t_0,x,y)-\gamma
\end{equation}
and due to the bounds (\ref{epszone}) and inequality
(\ref{almostopt}) we have
\begin{equation}
\sup_{\alpha\in\mathcal{A}}v^{\alpha}(t_0,x,y)-\gamma-C_{\varepsilon,\theta,T}\le
v^{\tilde\alpha}(t_0,x,y) - C_{\varepsilon,\theta,T}\le
v^{\varepsilon,\tilde\alpha}(t_0,x,y)\le v^{\tilde\alpha}(t_0,x,y) +
C_{\varepsilon,\theta,T}.
\end{equation}
Now, suppose again that
$v^{\varepsilon,\bar\alpha^\varepsilon}(t_0,x,y)=\sup_{\alpha\in\mathcal{A}}
v^{\varepsilon,\alpha}(t_0,x,y). $ That means
$v^{\varepsilon,\bar\alpha^\varepsilon}(t_0,x,y)\ge
v^{\varepsilon,\tilde\alpha}(t_0,x,y).$

So, \begin{eqnarray}
\sup_{\alpha\in\mathcal{A}}v^{\alpha}(t_0,x,y)-\gamma-C_{\varepsilon,\theta,T}\le
v^{\tilde\alpha}(t_0,x,y)-C_{\varepsilon,\theta,T}\le
v^{\varepsilon,\tilde\alpha}(t_0,x,y)\nonumber\\
\le v^{\varepsilon,\bar\alpha^\varepsilon}(t_0,x,y)\le
v^{\bar\alpha^\varepsilon}(t_0,x,y)+ C_{\varepsilon,\theta,T}.
\end{eqnarray}
Thus,
\begin{equation}
\sup_{\alpha\in\mathcal{A}}v^{\alpha}(t_0,x,y)\le
v^{\bar\alpha^\varepsilon}(t_0,x,y)+\gamma+2C_{\varepsilon,\theta,T},
\end{equation}
i.e. $\bar\alpha^\varepsilon$ is an $(\gamma+2C_{\varepsilon,\theta,T})-$optimal
strategy for $v^{\alpha}(t_0,x,y)$ or
$\delta=\gamma+2C_{\varepsilon,\theta,T}$ in this case.

Finally, in the case that $\sup_{\alpha\in\mathcal{A}}
v^{\varepsilon,\alpha}(t_0,x,y)$ is not attained, let us consider a $\kappa-$optimal
strategy $\tilde\alpha^\varepsilon$ such that
$v^{\varepsilon,\tilde\alpha^\varepsilon}(t_0,x,y)\le\sup_{\alpha\in\mathcal{A}}
v^{\varepsilon,\alpha}(t_0,x,y)-\kappa.$ Following the same
reasoning as previously (note that $\tilde\alpha$ is
$\gamma-$optimal for $v^\varepsilon$), we get
\begin{eqnarray}
\sup_{\alpha\in\mathcal{A}}v^{\alpha}(t_0,x,y)-\gamma-\kappa-C_{\varepsilon,\theta,T}\le
v^{\tilde\alpha}(t_0,x,y)-C_{\varepsilon,\theta,T}-\kappa\nonumber\\\le
v^{\varepsilon,\tilde\alpha}(t_0,x,y)-\kappa \le
v^{\varepsilon,\tilde\alpha^\varepsilon}(t_0,x,y)\le
v^{\tilde\alpha^\varepsilon}(t_0,x,y)+ C_{\varepsilon,\theta,T}.
\end{eqnarray}
Therefore,
\begin{equation}
\sup_{\alpha\in\mathcal{A}}v^{\alpha}(t_0,x,y)\le
v^{\tilde\alpha^\varepsilon}(t_0,x,y)+\gamma+\kappa+2C_{\varepsilon,\theta,T},
\end{equation}
i.e. $\tilde \alpha^\varepsilon$ is an
$(\gamma+\kappa+2C_{\varepsilon,\theta,T})-$optimal strategy for
$v^{\alpha}(t_0,x,y)$ or $\delta=\gamma+\kappa+2C_{\varepsilon,\theta,T}$. Since $\varepsilon, \gamma>0$ and $\kappa>0$ here can be chosen arbitrarily small, this implies the statement about a nearly optimal strategy as required. 
\end{proof}

Finally, a little bit about Markov strategies. Recall that Markov strategies are sufficient for the regularised problem (\ref{Yt-epsilon}), see \cite{kry}. Due to the previous result, we have also the following
\begin{Proposition}
Let assumptions (A$_0$) hold. Then for the equation (\ref{Yt}) Markov strategies are also sufficient. 

\end{Proposition}
\noindent
\begin{proof}
We leave the reader to consult  \cite{kry} about using Markov strategies. Since we deal everywhere with "first moment theory" and additionally
an optimal $\bar\psi$ can always be found in a bounded real interval, the class of Markov strategies is, indeed,  sufficient for the problem due to Theorem  \ref{Thm53}. \end{proof}
\noindent
Alternatively, we could refer directly to the calculus in the proof of Theorem \ref{Thm53}. Emphasize that Markov strategies here even for the degenerate system depend on $(x,y)$, not just on $x$ variable.

\subsection{Viscosity approach}

When working with viscosity solutions, there is no verification theorem that can be applied under the assumptions made in this paper (see Gozzi et al. \cite{GSZ2009} for the latest results on the verification theorem for viscosity solutions). One possible approach is to verify the optimality of the strategies using Monte Carlo simulations using the strategy calculated from the solution of the HJB equation and compare the value of the value function obtained by simulation to the one from the numerical scheme.

\section{ Concluding Remarks}

In this paper, we formulated a general mean-variance problem in continuous time that includes a functional with two terms: an integral that depends on the whole trajectory of the controlled process and a terminal time one. We interpreted the problem first as a terminal time problem, through the introduction of a coupled state process with an additional dimension and then transformed it into a superposition of a static and a dynamic optimization problem, where the latter is feasible for dynamic programming methods and for which we were able to write down an HJB equation. We proved existence and uniqueness of solutions both in viscosity sense and also in classical (Sobolev) sense. The advantage of the first approach is that there is no need to address the inherent degeneracy of the coupled state process, whereas numerical solutions can be employed to solve the problem and to even show optimality thought Monte Carlo simulations. A verification theorem is not readily applicable under the assumptions we use. 

When following the Sobolev approach, a regularisation of the state process is required. This has the advantage that a verification theorem can be obtained (through It\^o-Krylov's formula). We then showed that strategies obtained through this route are nearly optimal.    

~


\noindent {\bf Acknowledgments:}\\
The first author has beens partially supported by EPSRC grant EP/N013980/1 and the Alan Turing Institute EP/N510129/1. For the second author 
this study has been funded by the Russian Academic Excellence Project '5-100'.   The authors are also grateful to 
Jan Palczewski for valuable remarks.

\end{document}